\documentclass[a4paper,twoside,11pt]{article}

\usepackage{a4wide, fancyhdr, amsmath, amssymb, mathtools, yfonts, enumerate}
\usepackage{mathrsfs}
\usepackage{graphicx}
\usepackage{tikz}
\usepackage[all]{xy}
\usepackage[utf8]{inputenc}
\usepackage{amsthm}
\usepackage[english]{babel}
\usepackage{chngcntr}
\usepackage{ifthen}
\usepackage{calc}
\usepackage{hyperref}
\usepackage{authblk}
\numberwithin{equation}{section}
\counterwithin*{equation}{section}

\newcommand{\Z}{\mathbb{Z}}
\newcommand{\Q}{\mathbb{Q}}
\newcommand{\R}{\mathbb{R}}

\newcommand\FF{\mathbb{F}}

\newcommand\Gal{\mathrm{Gal}}

\DeclareMathOperator{\rank}{rank}

\newtheorem{lemma}{Lemma}[section]
\newtheorem{theorem}[lemma]{Theorem}
\newtheorem{prop}[lemma]{Proposition}

\theoremstyle{definition}

\title{\vspace{-\baselineskip}\sffamily\bfseries A new pointwise bound for $3$-torsion of class groups}
\author[1]{Stephanie Chan}
\author[2]{Peter Koymans}
\affil[1]{IST Austria}
\affil[2]{Utrecht University}
\date{\today}

\begin{document}
\maketitle

\begin{abstract}
Ellenberg--Venkatesh proved in 2007 that $h_3(d) \ll_\epsilon |d|^{1/3 + \epsilon}$, where $h_3(d)$ denotes the size of the $3$-torsion of the class group of $\Q(\sqrt{d})$. We improve this bound to $h_3(d) \ll_\epsilon |d|^{\kappa + \epsilon}$ with $\kappa \approx 0.3193 \cdots$. We also combine our methods with work of Heath-Brown--Pierce to give new bounds for average $\ell$-torsion of real quadratic fields.
\end{abstract}

\section{Introduction}
Throughout this paper, we let $d$ be a squarefree integer and we let $\ell$ be a prime number. We denote by $h_\ell(d) := \# \mathrm{Cl}(\Q(\sqrt{d}))[\ell]$ the size of the $\ell$-torsion of the class group, and by $h(d)$ the size of the full class group. The Brauer--Siegel theorem provides the classical bound
\begin{align}
\label{ePointwise}
h_\ell(d) \leq h(d) \ll_\epsilon |d|^{1/2 + \epsilon}.
\end{align}
The bound \eqref{ePointwise} is frequently referred to as the trivial bound for $\ell$-torsion, as it does not make use at all of the special structure of the $\ell$-torsion of the class group. Although bounds for $h_\ell(d)$ are of great interest both inside and outside arithmetic statistics, the current state of the art is that no conditional pointwise upper bounds better than \eqref{ePointwise} are known outside the cases $\ell \in \{2, 3\}$. The case $\ell = 2$ is classical, going back to Gauss, while the case $\ell = 3$ is treated in works by Pierce \cite{MR2254390}, Helfgott--Venkatesh \cite{HV} and Ellenberg--Venkatesh \cite{EV}. The current record is the bound $h_3(d) \ll_\epsilon |d|^{1/3 + \epsilon}$, see \cite[Proposition 2]{EV}.

\begin{theorem}
\label{theorem:main}
There exists a constant $C > 0$ such that for all squarefree integers $d$
$$
h_3(d) \leq C \cdot |d|^{0.3194}.
$$
\end{theorem}


In fact, we will show the stronger statement
$$
h_3(d) \ll_\epsilon |d|^{\kappa + \epsilon} 
$$
with $\kappa =\kappa(\gamma)\approx 0.3193 \cdots$, where $\kappa$ is defined in \eqref{eq:kapform} and $\gamma$ is the unique solution to $G(\gamma) = 0$; see \eqref{eq:Groot} for the definition of the function $G(\gamma)$. Conditional on deep but standard conjectures, it was previously shown that $h_3(d) \ll_\epsilon |d|^{1/4 + \epsilon}$, see the work of Wong \cite{Wong} and of Shankar--Tsimerman \cite{ShankarTsi}. For more general number fields, pointwise bounds are known in a limited number of settings \cite{BSTTTZ, EV, KlunersWang, LOZ, PTW2, Wang-Nilpotent, Wang}.

Our proof strategy synthesizes ideas from many of the previous works in the area, in particular we rely on ideas of Ellenberg--Venkatesh \cite{EV}, Heath-Brown--Pierce \cite{HBP}, Frei--Widmer \cite{FW}, Koymans--Thorner \cite{KT} and Helfgott--Venkatesh \cite{HV}. The starting point of the field is (independent) insights of Michel and Soundararajan that $h(d)/h_\ell(d)$ can be lower bounded in the presence of sufficiently many small split primes. This was first set on firm ground in a very influential lemma of Ellenberg--Venkatesh \cite[Lemma 3]{EV}.

In order to progress further, it was soon realized that further improvements to the aforementioned ``Ellenberg--Venkatesh lemma'' were pivotal. This is typically done by allowing for more small split primes, but this comes at the price that one needs to keep track of relations (inside the class group) between the resulting small split primes. For imaginary quadratics, this was codified in Heath-Brown--Pierce \cite[Proposition 2.1]{HBP} by using a Cauchy--Schwarz type argument, while this was done in a slightly different form for general number fields in Frei--Widmer \cite[Proposition 2.1]{FW}. A key insight of Koymans--Thorner \cite{KT} is that both approaches may be fruitfully combined to give a very flexible Ellenberg--Venkatesh lemma. 

We start by specializing the method of Koymans--Thorner \cite[Theorem 3.3]{KT} to real quadratic fields in Section \ref{sKT}. The extra flexibility from this result allows us to take small split primes up to a bound of our choosing at the price of having to bound relations between said primes. In order to find such small split primes, we use a trick also due to Ellenberg--Venkatesh \cite[Proposition 2]{EV}, namely that the $3$-rank of $\Q(\sqrt{d})$ and $\Q(\sqrt{-3d})$ are close by a classical reflection principle of Scholz. By looking at the splitting of primes $p \equiv 2 \bmod 3$ in the biquadratic field $\Q(\sqrt{d}, \sqrt{-3d})$, one deduces that there must be many small split primes in either the field $\Q(\sqrt{d})$ or $\Q(\sqrt{-3d})$.

It remains to obtain a good control over the relations between small primes. It is at this point that we return to the original argument of Helfgott--Venkatesh \cite{HV}. This work combines repulsion of integer points with sphere packing arguments from \cite{KL} to give strong bounds for integral points on elliptic curves. We present their argument in the special case of interest to us in Section \ref{sHV}, and we will finish the proof of Theorem \ref{theorem:main} in Section \ref{sMainT}.

Our techniques are also able to make progress on average $\ell$-torsion. This problem is significantly more tractable than pointwise bounds: for example the large sieve provides many small split primes for almost all fields. For the remaining fields, one can use equation \eqref{ePointwise}. Various types of average bounds are available in \cite{EPW, FW2, FW, Widmer} with the state of the art being the results in Lemke Oliver--Smith \cite{LOS}. For imaginary quadratic fields, the strongest currently available average bounds are due to Heath-Brown--Pierce \cite[Theorem 1.1]{HBP}. Our next result extends their work to real quadratic fields.

\begin{theorem}
\label{t2}
Let $\ell \geq 5$ be a prime and let $\epsilon > 0$. Then there exists $C > 0$ such that for all $X \geq 1$ there holds
$$
\sum_{\substack{0 < d < X \\ d \textup{ squarefree}}} h_\ell(d) \leq C X^{\frac{3}{2} - \frac{3}{2\ell + 2} + \epsilon}.
$$
\end{theorem}

Theorem \ref{t2} will fall as a consequence of Theorem \ref{tEV} and a straightforward adaptation of \cite[Proposition 2.3]{HBP}. This argument will be carried out in Section \ref{sBonusT}. 


\subsection*{Acknowledgements}
PK gratefully acknowledges the support of the Dutch Research Council (NWO) through the Veni grant ``New methods in arithmetic statistics''. 

\section{Bounding torsion in real quadratic fields}
\label{sKT}
Let $K$ be a number field. Let $\mathcal{P}_K^{(1)}$ be the prime ideals of $K$ with inertia and residue field degree $1$. We define 
$$
\pi_K^{(1)}(Z) := \{\mathfrak{p} \in \mathcal{P}_K^{(1)} : N_{K/\Q}(\mathfrak{p}) \leq Z\}.
$$
If $K = \Q(\sqrt{d})$ is a real quadratic number field for some squarefree $d > 0$, we write $O_d$ for its ring of integers, $\sigma$ for a generator of $\Gal(K/\Q)$ and $\pi_K^{(1)}(Z) = \pi_d(Z)$. We introduce for every prime number $\ell$ and every real number $Z > 0$ the set
$$
S_\ell(d, Z) := \left\{\beta = u + v \sqrt{d} \in O_d : 
\begin{array}{l} 
|u| \leq 2e^{3/2} Z^\ell, \quad |v| \leq 2e^{3/2} Z^\ell d^{-1/2} \\ 
\beta O_d = (\mathfrak{p}_1 \sigma(\mathfrak{p}_2))^{\ell} \text{ for some distinct } \mathfrak{p}_1, \mathfrak{p}_2 \in \pi_d(Z)
\end{array}
\right\}.
$$
We will frequently use the abbreviation $\mathrm{Cl}_d := \mathrm{Cl}(\Q(\sqrt{d}))$. Our next theorem is inspired by the argument from \cite[Theorem 3.3]{KT}.

\begin{theorem}
\label{tEV}
Let $\ell \geq 2$ and $Z > 0$. Then we have for all squarefree $d > 1$ with $|\pi_d(Z)| > 0$
$$
|\mathrm{Cl}_d[\ell]| \ll \frac{d^{1/2} \log d}{|\pi_d(Z)|} + \frac{d^{1/2} \log d}{|\pi_d(Z)|^2} \cdot |S_\ell(d, Z)|.
$$
The implied constant is absolute.
\end{theorem}

\begin{proof}
Let $R_d$ be the regulator of $\Q(\sqrt{d})$, and define $A := \mathrm{Cl}_d/\mathrm{Cl}_d[\ell]$. By work of Louboutin \cite{Louboutin} we have
$$
|\mathrm{Cl}_d[\ell]| \cdot |A| \cdot R_d = |\mathrm{Cl}_d| \cdot R_d \ll d^{1/2} \log d.
$$
Therefore it suffices to show that
\begin{align}
\label{eClaim}
|A| \cdot R_d \geq \left(\frac{1}{|\pi_d(Z)|} + \frac{|S_\ell(d, Z)|}{|\pi_d(Z)|^2}\right)^{-1}
\end{align}
for $d > 100$. Write $v_1$ and $v_2$ for the two real places of $\Q(\sqrt{d})$, and consider the homomorphism $\varphi: \Q(\sqrt{d})^\ast \rightarrow \mathbb{R}^2$ given by
$$
\varphi(\alpha) = \left(\log |\alpha|_{v_1}, \log |\alpha|_{v_2}\right).
$$
We now consider the $1$-dimensional subspace 
$$
V_0 = \{(x_1, x_2) \in \mathbb{R}^2 : x_1 + x_2 = 0\}
$$
of $\mathbb{R}^2$. Then Dirichlet's unit theorem is simply that $\mathcal{L} := \varphi(O_d^\ast)$ is a rank $1$ lattice inside $V_0$. Let $\mathbf{u}$ be the shortest non-zero vector inside $\mathcal{L}$ with positive first coordinate, and write $\mathbf{u} = (y, -y)$ with $y > 0$. Then we have $R_d = y$. With this notation set, note that
$$
\mathcal{F} = \{\lambda \mathbf{u} : 0 \leq \lambda < 1\}
$$
is a fundamental domain of $\mathcal{L}$. In terms of the basis $\mathbf{u} = (y, -y)$ and $\mathbf{p} := (1, 1)$ of $\mathbb{R}^2$, we calculate that the map $\varphi$ is explicitly given by
\begin{align}
\label{eBasisChange}
\varphi(\alpha) &= \left(\frac{y \log |\alpha|_{v_1} - y \log |\alpha|_{v_2}}{2y^2}\right) \mathbf{u} + \left(\frac{\log |\alpha|_{v_1} + \log |\alpha|_{v_2}}{2}\right) \mathbf{p} \nonumber \\
&= \left(\frac{\log |\alpha|_{v_1} - \log |\alpha|_{v_2}}{2R_d}\right) \mathbf{u} + \left(\frac{\log |N_{K/\Q}(\alpha)|}{2}\right) \mathbf{p}.
\end{align}
We also see that for all $\alpha \in \Q(\sqrt{d})^\ast$, there exists $\epsilon_\alpha \in O_d^\ast$ (unique up to multiplication by an element of $\ker(\varphi) = \{1, -1\})$, $\mathbf{v}_\alpha \in \mathcal{F}$ and $y_\alpha \in \mathbb{R}$ such that
$$
\varphi(\epsilon_\alpha \alpha) = \mathbf{v}_\alpha + y_\alpha \mathbf{p}.
$$
For real quadratic fields, it is elementary to show that $R_d \geq \log\left(\frac{1}{2} \sqrt{d - 4} + \frac{1}{2} \sqrt{d}\right) > 2$ for $d > 100$ (see \cite[Eq. (1.1)]{Regulator}), which we henceforth assume. Define $n := \lfloor R_d \rfloor \geq 2$. We now subdivide $\mathcal{F}$ into $n$ intervals of equal length, i.e.~we define for each $i \in \{0, \dots, n - 1\}$ the set
$$
\mathcal{F}_i = \left\{\lambda \mathbf{u} : \frac{i}{n} \leq \lambda < \frac{i + 1}{n}\right\}.
$$
We fix a complete set of representatives $\mathfrak{b}_1, \dots, \mathfrak{b}_h$ for $\mathrm{Cl}_d$. Then for each prime ideal $\mathfrak{p} \in \pi_d(Z)$, there exists a unique class $\mathfrak{b}_j$ such that $\mathfrak{b}_j \mathfrak{p}^\ell$ is principal. We denote by $a$ the image of $\mathfrak{b}_j$ inside $A$. We also fix a choice of $\alpha$ with $\mathfrak{b}_j \mathfrak{p}^\ell = (\alpha)$ and we let $i$ be the unique integer such that $\mathbf{v}_\alpha \in \mathcal{F}_i$. Then we have constructed a map
$$
\psi: \pi_d(Z) \rightarrow A \times \{0, \dots, n - 1\}, \quad \quad \psi(\mathfrak{p}) = (a, i).
$$
The Cauchy--Schwarz inequality states
$$
|\pi_d(Z)| = \sum_{(a, i)} |\psi^{-1}(a, i)| \leq \left(\sum_{\substack{(a, i) \\ \psi^{-1}(a, i) \neq \varnothing}} 1\right)^{1/2} \left(\sum_{\substack{(a, i) \\ \psi^{-1}(a, i) \neq \varnothing}} |\psi^{-1}(a, i)|^2\right)^{1/2}.
$$
Hence we have
$$
|A| \cdot R_d \geq |A| \cdot n \geq \sum_{\substack{(a, i) \\ \psi^{-1}(a, i) \neq \varnothing}} 1 \geq \frac{|\pi_d(Z)|^2}{\sum\limits_{\substack{(a, i) \\ \psi^{-1}(a, i) \neq \varnothing}} |\psi^{-1}(a, i)|^2}.
$$
Going back to equation \eqref{eClaim}, we thus have to show that 
$$
\sum_{\substack{(a, i) \\ \psi^{-1}(a, i) \neq \varnothing}} |\psi^{-1}(a, i)|^2 \leq |\pi_d(Z)| + |S_\ell(d, Z)|.
$$
Let us now split the contribution of $|\psi^{-1}(a, i)|^2 = |\psi^{-1}(a, i) \times \psi^{-1}(a, i)|$ into two pieces. Firstly, the diagonal elements $(\mathfrak{p}, \mathfrak{p})$ contribute exactly $|\pi_d(Z)|$. Secondly, for off-diagonal elements $(\mathfrak{p}_1, \mathfrak{p}_2)$ with $\mathfrak{p}_1 \neq \mathfrak{p}_2$, we can do the following.

Since $\mathfrak{p}_1$ and $\mathfrak{p}_2$ have the same class $a$ in $\mathrm{Cl}_d/\mathrm{Cl}_d[\ell]$, it follows that $\mathfrak{p}_1/\mathfrak{p}_2$ is $\ell$-torsion. Hence $\mathfrak{p}_1^\ell$ and $\mathfrak{p}_2^\ell$ are equivalent in the class group. Hence there is a class $\mathfrak{b}_j$ and elements $\alpha_1, \alpha_2$ with $\mathfrak{b}_j \mathfrak{p}_1^\ell = (\alpha_1)$ and $\mathfrak{b}_j \mathfrak{p}_2^\ell = (\alpha_2)$ and moreover $\mathbf{v}_{\alpha_1}, \mathbf{v}_{\alpha_2} \in \mathcal{F}_i$. We now define
$$
\beta := \frac{\epsilon_{\alpha_1} \alpha_1 \sigma\left(\epsilon_{\alpha_2} \alpha_2\right)}{N_{K/\Q}(\mathfrak{b}_j)} = u + v \sqrt{d} \in O_d,
$$
where we recall that $\sigma$ is the unique non-trivial element of $\Gal(\Q(\sqrt{d})/\Q)$. We expand
\begin{align*}
\varphi(\epsilon_{\alpha_1} \alpha_1) &= \mathbf{v}_{\alpha_1} + y_{\alpha_1} \mathbf{p} \\
\varphi(\epsilon_{\alpha_2} \alpha_2) &= \mathbf{v}_{\alpha_2} + y_{\alpha_2} \mathbf{p} \\
\varphi(\sigma(\epsilon_{\alpha_2} \alpha_2)) &= -\mathbf{v}_{\alpha_2} + y_{\alpha_2} \mathbf{p} \\
\varphi(N_{K/\Q}(\mathfrak{b}_j)) &= \log(N_{K/\Q}(\mathfrak{b}_j)) \mathbf{p}.
\end{align*}
By \eqref{eBasisChange}, we observe that $y_{\alpha_i} = \log(|N_{K/\Q}(\alpha_i)|)/2$.
From this, we conclude that
$$
\varphi(\beta) = \mathbf{v}_{\alpha_1} - \mathbf{v}_{\alpha_2} + \frac{\ell \log(N_{K/\Q}(\mathfrak{p}_1 \mathfrak{p}_2))}{2} \mathbf{p}.
$$
Since $\mathbf{v}_{\alpha_1}, \mathbf{v}_{\alpha_2} \in \mathcal{F}_i$, it follows that $\mathbf{v}_{\alpha_1} - \mathbf{v}_{\alpha_2} = \delta \mathbf{u}$ with $|\delta| \leq \frac{1}{n}$. By \eqref{eBasisChange}, this gives
$$
\left|\log |\beta|_{v_1} - \log |\beta|_{v_2}\right| = 2R_d |\delta| \leq \frac{2R_d}{n} \leq 3
$$
since $n = \lfloor R_d \rfloor$ and $R_d \geq 2$. Recalling that $\beta = u + v \sqrt{d}$, we obtain the simultaneous inequalities
$$
e^{-3} |u + v \sqrt{d}| \leq |u - v \sqrt{d}| \leq e^3 |u + v \sqrt{d}|, \quad \quad |u^2 - dv^2| \leq Z^{2 \ell}.
$$
This implies $|u - v \sqrt{d}| \leq e^{3/2} Z^\ell$ and $|u + v \sqrt{d}| \leq e^{3/2} Z^\ell$. We conclude that $|u| \leq 2e^{3/2} Z^\ell$ and $|v| \leq 2e^{3/2} Z^\ell d^{-1/2}$, as desired.
\end{proof}

\section{Integral points on elliptic curves with bounded height}
\label{sHV}
Let $D$ be a non-zero integer and let $E_D: y^2 = x^3 + D$. We write 
$$
E_D(\Z) \coloneqq \{(x, y) \in \Z^2 : y^2 = x^3 + D\}. 
$$
Write $\overline{\Q}$ for the algebraic closure of $\Q$. Denote the canonical height by
\[
\hat{h}: E(\overline{\Q})\rightarrow \R_{\geq 0}, \qquad \hat{h}(P) \coloneqq \frac{1}{2} \lim_{n\rightarrow\infty} \frac{h(x(2^nP))}{4^n},
\]
where $h$ denotes the absolute logarithmic height on $\overline{\Q}$.

To obtain an upper bound for the number of integral points on $E_D$ with bounded height, we apply \cite[Theorem 3.8]{HV}, which works in a much more general setting. For the convenience of the reader, we give a more direct proof specialized to our case based on \cite{Helfgott}. Define for $0 < t < 1$ and  $0 < \gamma < 1$ the functions
\begin{equation}
\label{eq:deffg}
f(t) \coloneqq \frac{1 + t}{2t} \log\frac{1 + t}{2t} - \frac{1 - t}{2t} \log\frac{1 - t}{2t}\quad\text{and}\quad g(\gamma) \coloneqq f\left(\frac{1}{2} \sqrt{(1 + \gamma) (3 - \gamma)}\right).
\end{equation}

\begin{theorem}
\label{theorem:intbound}
There exists $C > 0$ such that the following holds. Let $\epsilon > 0$ and $0 < \gamma < 1$. Let $D$ be an integer such that the sixth-power free part of $D$ is at least $\exp(C/\epsilon)$. Then we have for all $Z \geq |D|^{\frac{1}{6}}$
\begin{align*}
\#\left\{(x, y)\in E_D(\Z): |x|\leq Z^2\right\}
&\leq \#\left\{P\in E_D(\Z): \hat{h}(P) \leq \log Z + C\right\}\\
&\ll_{\epsilon, \gamma} Z^{\gamma + \epsilon} \times \exp\left(\left(g(\gamma) + \epsilon\right) \rank_{\Z} E_D(\Q)\right).
\end{align*}
\end{theorem}

We give the following estimate for the canonical height.

\begin{lemma}
\label{lemma:heightest}
Let $P \in E_D(\Q)$ and write $x(P) = a/b$ for some integers $a$ and $b \neq 0$ not necessarily coprime. Then we have
\[
\hat{h}(P) = \frac{1}{6} \log \max\{|a^3|, |b^3D|\} - \frac{1}{6} \log \gcd(a^3, b^3D) + O(1),
\]
where the implied constant is absolute.
\end{lemma}

\begin{proof}
Recall from \cite[Chapter~VIII, Theorem 9.3]{SilvermanAEC} that the canonical height does not depend on the model of the curve, and that $|\hat{h}(P) - \frac{1}{2}h(x(P))|$ is bounded by some constant depending only on the model of the elliptic curve. First, we make use of the isomorphism that sends $P = (x, y) \in E_D(\Q)$ to $(x/D^{\frac{1}{3}},y/D^{\frac{1}{2}})\in E_1(\overline{\Q})$ to get $\hat{h}(P) = \hat{h}((x/D^{\frac{1}{3}},y/D^{\frac{1}{2}}))$. Second, we convert to the naive height $h$ on $E_1(\overline{\Q})$ to obtain
\[
2\hat{h}(P) = h\left(\frac{a}{bD^{\frac{1}{3}}}\right) + O(1) = \frac{1}{3} \log \frac{\max\{|a^3|, |b^3D|\}}{\gcd(a^3, b^3D)} + O(1),
\]
where the implied constant depends only on $E_1$ but not on $D$.
\end{proof}

The first inequality in Theorem~\ref{theorem:intbound} follows from Lemma~\ref{lemma:heightest}, so we will henceforth focus on bounding $\#\{P \in E_D(\Z): \hat{h}(P) \leq \log Z + C\}$. Following the approach in \cite{HV}, we pick a large auxiliary prime $p$ and we proceed to partition our integral points according to $(x, y) \bmod p$.

\begin{lemma}
\label{lemma:gapprinciple}
There exists a constant $C>0$ such that the following holds. Let $m$ be a positive integer coprime to $D$. Let $P_1 = (x_1, y_1),\ P_2 = (x_2, y_2) \in E_D(\Z)$ be such that $x_1\neq x_2$, $(x_1, y_1) \equiv (x_2, y_2) \bmod m$ and $\gcd(x_1^3, D) = \gcd(x_2^3, D)$. Then
\[
\hat{h}(P_1 + P_2) \leq\hat{h}(P_1) + \hat{h}(P_2)+ \max_{i\in \{1, 2\}} \hat{h}(P_i)  - \log m +C.
\]
\end{lemma}

\begin{proof}
From the addition formula for points in $E_D(\Q)$, we have
\[
x(P_1 + P_2) = \left(\frac{y_1 - y_2}{x_1 - x_2}\right)^2 - x_1 - x_2
= \frac{x_1^2x_2 + x_1x_2^2 - 2y_1y_2 + 2D}{(x_1 - x_2)^2}.
\]
To bound $\hat{h}(P_1 + P_2)$, we apply Lemma~\ref{lemma:heightest} with $a = x_1^2x_2 + x_1x_2^2 - 2y_1y_2 + 2D$ and $b = (x_1 - x_2)^2$, and we compute
\begin{align*}
\max\{|a^3|,|b^3D|\}
&=\max\{|x_1^2x_2+x_1x_2^2-2y_1y_2+2D|,|D|^{\frac{1}{3}}(x_1-x_2)^2\}^3\\
&\ll \max\{|x_1^2x_2|,|x_1x_2^2|,|y_1y_2|,|D|,|D|^{\frac{1}{3}}x_1^2,|D|^{\frac{1}{3}}x_2^2\}^3\\
&\ll \max_{\{i,j\}=\{1,2\}}\max\{|x_i|^3,|D|\}^2 \max\{|x_j|^3,|D|\}.
\end{align*}
Since $(x_1, y_1) \equiv (x_2, y_2) \bmod m$, we see that $m \mid \gcd(y_1 - y_2, x_1 - x_2)$, and therefore we plainly have $m^2 \mid b$. But we also have
$$
a = (y_1 - y_2)^2 - (x_1 + x_2) (x_1 - x_2)^2 \equiv 0 \bmod m^2.
$$
We conclude that $m^2 \mid \gcd(a, b)$.

Let $g \coloneqq \gcd(x_1^3, D) = \gcd(x_2^3, D)$. One checks that $g^3 \mid (x_i^3)^2 x_j^3 = (x_i^2 x_j)^3$, so $g \mid x_i^2x_j$. Moreover, we have $g \mid y_1^2$ and $g \mid y_2^2$, hence $g^2 \mid y_1^2 y_2^2$ and $g \mid y_1y_2$. We deduce that $g \mid x_1^2x_2 + x_1x_2^2 - 2y_1y_2 + 2D$. Moreover, we have $g^3 \mid (x_1 - x_2)^6 D$ by expanding $(x_1 - x_2)^6$ with the binomial theorem and using the previously obtained divisibilities. Since $m$ and $g$ are coprime, we can combine our divisibility conditions into $(m^2g)^3 \mid \gcd(a^3, b^3D)$.

Now it follows from Lemma~\ref{lemma:heightest} that
\begin{align*}
\hat{h}(P_1 + P_2)
&\leq \max_{\{i, j\} = \{1, 2\}} \left\{2\hat{h}(P_i) + \hat{h}(P_j) + \frac{1}{2}\log g\right\} - \frac{1}{2} \log g - \log m + C \\
&\leq \max_{\{i, j\} = \{1, 2\}} \left\{2\hat{h}(P_i) + \hat{h}(P_j)\right\} - \log m + C
\end{align*}
as required. 
\end{proof}

\begin{prop}[{\cite{KL}}]
\label{prop:spherepack}
Let $A(n,\theta)$ be the maximal number of points in $\R^n$ such that the angle between any two points and the origin is at least $\theta$. Let $f$ be the function as defined in \eqref{eq:deffg}. Then for $0<\theta<\pi/2$, we have
\[
\log A(n,\theta) \leq n\left( f(\sin\theta) + o(1)\right),
\]
where the convergence as $n \rightarrow \infty$ is uniform for $\theta$ within any closed subinterval of $(0,\pi/2)$.
\end{prop}

It follows from \cite{Silvermanlowerbounds} that there exist absolute constants $c_1, c_2 > 0$ satisfying
\[
\hat{h}(P) \geq c_1\log |D'| - c_2,
\]
where $D'$ is the sixth-power free part of $D$ and $P$ is any non-torsion in $E_D(\Q)$. 

We take $\gamma$ and $\epsilon > 0$ as in Theorem \ref{theorem:intbound}. Note that we may assume without loss of generality that $\epsilon \leq 1/4$. Let $c_1\log |D'| - c_2\leq B\leq \log Z$ be a number and choose a prime $p \nmid D$ such that $B(\gamma-\epsilon)\leq\log p\leq B(\gamma+\epsilon)$, which is possible as long as $\epsilon\log |D'|$ is sufficiently large, so that $\epsilon B$ is sufficiently large. Suppose $P_1$ and $P_2$ satisfy the assumptions of Lemma~\ref{lemma:gapprinciple}, and are such that 
\begin{align}
\label{eDyadic}
(1 - \epsilon) B < \hat{h}(P) \leq B.
\end{align}
Then applying the bound we obtained in Lemma~\ref{lemma:gapprinciple}, we see that the angle $\theta$ between $P_1$ and $P_2$ obeys
\begin{align*}
\cos\theta& = \frac{\hat{h}(P_1+P_2)-\hat{h}(P_1)-\hat{h}(P_2)}{2\sqrt{\hat{h}(P_1)\hat{h}(P_2)}}\leq \frac{\max\left\{\hat{h}(P_1),\hat{h}(P_2)\right\} - \log p + C}{2\sqrt{\hat{h}(P_1)\hat{h}(P_2)}} \\
&\leq\frac{1}{2}\left( \frac{1}{1 - \epsilon} - \frac{\log p - C}{B}\right)
\leq \frac{1}{2}(1 - \gamma) + C_1 \epsilon
\end{align*}
for some absolute constant $C_1 > 0$. Then since $0 < \theta < \pi/2$, we obtain
\[
\sin\theta = \sqrt{1 - \cos^2\theta} \geq \frac{1}{2} \sqrt{(1+\gamma) (3 - \gamma)} - C_2\epsilon
\]
for some absolute constant $C_2 > 0$. Note that when $0 < \gamma < 1$, we have
\begin{align}
\label{eUseful}
\frac{\sqrt{3}}{2} \leq \frac{1}{2}\sqrt{ (1+\gamma) (3 - \gamma)} \leq 1 \quad \text{and} \quad 1 \leq \exp\left(g(\gamma)\right)\leq \frac{\left(7 + 4 \sqrt{3}\right)^{ 1/\sqrt{3}}}{2\sqrt{3}}.
\end{align}
Fix $(x_0, y_0) \in \Z/p\Z \times \Z/p\Z$ such that $y_0^2 = x_0^3 + D$, and fix $g_0\mid D$. Then the sphere packing bound in Proposition~\ref{prop:spherepack} gives
\begin{multline*}
\#\left\{P\in E_D(\Z) : (1 - \epsilon) B < \hat{h}(P)\leq B,\ P\equiv (x_0, y_0) \bmod p,\ \gcd(x(P)^3, D) = g_0\right\}\\
\leq \exp\left(r\left(f\left( \frac{1}{2}\sqrt{ (1+\gamma) (3 - \gamma)}\right) + C_3 \epsilon\right)\right)
= \exp\left(\left(g(\gamma) + C_3 \epsilon\right)r\right)
\end{multline*}
where $r \coloneqq \rank_{\Z} E_D(\Q)$ and where $C_3 > 0$ is an absolute constant. Since fixing $x_0\bmod p$ also fixes $y_0\bmod p$ up to sign, there are $O(p)$ many classes of $(x_0, y_0) \bmod p$. Summing over all $(x_0, y_0) \bmod p$ and all $g_0\mid D$, we have for sufficiently large $|D|$
\[
\#\left\{P\in E_D(\Z):(1-\epsilon)B<\hat{h}(P)\leq B\right\} \ll_{\epsilon} |D|^{\epsilon} \exp( B(\gamma+\epsilon))\exp\left(\left(g(\gamma) + C_3 \epsilon\right)r\right).
\]
Finally, summing over intervals of the shape \eqref{eDyadic} up to $\log Z + C$, we obtain
\[
\#\left\{P \in E_D(\Z) : \hat{h}(P) \leq \log Z + C\right\} \ll_{\epsilon, \gamma} \log\log Z \times |D|^\epsilon \times Z^{\gamma + \epsilon} \times \exp\left(\left(g(\gamma) + C_3 \epsilon\right)r\right).
\]
Then Theorem~\ref{theorem:intbound} follows from rescaling $\epsilon$ and $Z$.

\section{Proof of main theorem}
\label{sMainT}
The goal of this section is to prove Theorem~\ref{theorem:main}. We begin by recalling the corresponding version of Theorem~\ref{tEV} for imaginary quadratic fields from \cite[Proposition 2.1]{HBP}. For a more general statement we refer the reader to \cite[Theorem 3.3]{KT}. For $d < 0$, define
\[
S_\ell(d, Z) \coloneqq \left\{\begin{array}{l} u + v \sqrt{d} \in O_d \\ u \neq 0, v \neq 0 \end{array} : (p_1p_2)^\ell = u^2 - dv^2\text{ for some primes } p_1, p_2 \leq Z\right\}.
\]

\begin{theorem}
\label{theorem:neg}
There exists $C > 0$ such that for all primes $\ell$, for all $Z > 0$ and for all squarefree $d < -1$ with $|\pi_d(Z)| > 0$
$$
|\mathrm{Cl}_d[\ell]| \leq \frac{C |d|^{1/2} \log |d|}{|\pi_d(Z)|} + \frac{C |d|^{1/2} \log |d|}{|\pi_d(Z)|^2} \cdot |S_\ell(d, Z)|.
$$
\end{theorem}

\begin{proof}
This is similar to \cite[Proposition 2.1]{HBP}, although it does not explicitly follow from the work there. Therefore we have opted to prove this from scratch. Consider the exact sequence
$$
0 \rightarrow \mathrm{Cl}_d[\ell] \rightarrow \mathrm{Cl}_d \xrightarrow{\cdot \ell} \ell \mathrm{Cl}_d \rightarrow 0.
$$
By the class number formula and work of Louboutin \cite{Louboutin}, we deduce that
\begin{align}
\label{eBS}
|\mathrm{Cl}_d[\ell]| = \frac{|\mathrm{Cl}_d|}{|\ell \mathrm{Cl}_d|} \ll \frac{|d|^{1/2} \log |d|}{|\ell \mathrm{Cl}_d|}.
\end{align}
We now give a lower bound for $|\ell \mathrm{Cl}_d|$. Consider the map $\psi: \pi_d(Z) \rightarrow \ell \mathrm{Cl}_d$, which sends a prime ideal $\mathfrak{p}$ to the class $[\mathfrak{p}^\ell]$. The Cauchy--Schwarz inequality gives
$$
|\pi_d(Z)| = \sum_{a \in \ell \mathrm{Cl}_d} 1 \cdot |\psi^{-1}(a)| \leq \left(\sum_{a, \ \psi^{-1}(a) \neq \varnothing} 1\right)^{1/2} \left(\sum_{a, \ \psi^{-1}(a) \neq \varnothing} |\psi^{-1}(a)|^2\right)^{1/2},
$$
which yields
\begin{align}
\label{eEllLower}
|\ell \mathrm{Cl}_d| \geq \sum_{a, \ \psi^{-1}(a) \neq \varnothing} 1 \geq \frac{|\pi_d(Z)|^2}{\sum_{a, \ \psi^{-1}(a) \neq \varnothing} |\psi^{-1}(a)|^2}.
\end{align}
Inside the Cartesian products $|\psi^{-1}(a)|^2 = |\psi^{-1}(a) \times \psi^{-1}(a)|$ as $a$ varies, the diagonal elements $(\mathfrak{p}, \mathfrak{p})$ contribute $\pi_d(Z)$. We map the off-diagonal elements $(\mathfrak{p}_1, \mathfrak{p}_2)$ into $S_\ell(d, Z)$ by taking any generator of the principal ideal $\mathfrak{p}_1^\ell \sigma(\mathfrak{p}_2)^\ell$, where $\sigma$ is the unique non-trivial element of $\Gal(\Q(\sqrt{d})/\Q)$. This map is at most $2$-to-$1$, as only the pairs $(\mathfrak{p}_1, \mathfrak{p}_2)$ and $(\sigma(\mathfrak{p}_2), \sigma(\mathfrak{p}_1))$ map to the same ideal. Combining equations \eqref{eBS} and \eqref{eEllLower} then gives the theorem.
\end{proof}

We follow the approach in \cite[Proposition~3.4]{EV} in order to control $|\pi_d(Z)|$. Let $D\notin\{1,-3\}$ be a squarefree integer. Any prime not dividing $6D$ which is inert in $\Q(\sqrt{-3})$ splits in $\Q(\sqrt{D})$ or $\Q(\sqrt{-3D})$. The number of prime factors dividing $6D$ is bounded by $\ll \log D$, so at least one of $\pi_{D}(Z)$ and $\pi_{-3D}(Z)$ has size $\gg Z/\log Z$ provided that, say, $Z$ is at least a fixed power of $D$. We henceforth fix $d \in \{D, -3D\}$ such that $|\pi_{d}(Z)| \gg Z/\log Z$. By Scholz’s reflection principle, we have
\[
\left|\dim_{\FF_3}\mathrm{Cl}_D[3]- \dim_{\FF_3}\mathrm{Cl}_{-3D}[3]\right| \leq 1.
\] 
Therefore it suffices to bound $\dim_{\FF_3} \mathrm{Cl}_d[3]$ for this $d$ we fixed. Assume that
\begin{equation}
\label{eq:lbZ}
Z \geq |d|^{\frac{1}{6}}.
\end{equation}
From either Theorem~\ref{tEV} or Theorem~\ref{theorem:neg}, we have 
\begin{equation}
\label{eq:cl1}
|\mathrm{Cl}_d[3]| \ll_{\epsilon}  Z^{\epsilon}\left(\frac{|d|^{1/2} }{Z} + \frac{|d|^{1/2} }{Z^2} \cdot |S_3(d, Z)|\right).
\end{equation}
Given $\beta\in S_3(d, Z)$, we set $u' := 2u \in \Z_{\neq 0}$ and $v' := 2v \in \Z_{\neq 0}$. Then there exist primes $p_1, p_2$ such that $4(p_1p_2)^3 = u'^2 - dv'^2$. This corresponds to a point $(4p_1p_2, 4u') \in E_{16dv'^2}(\Z)$, with $p_1p_2 \leq Z^2$ and $|d| v'^2 \leq 16 e^3 Z^6$ (this bound is automatic for $d < 0$ and guaranteed by Theorem \ref{tEV} for $d > 0$). We apply Theorem~\ref{theorem:intbound} to bound the points on each $E_{16dv'^2}$, which yields the estimate
\begin{align}
\label{eS3dZ}
|S_3(d, Z)| \ll_{\epsilon, \gamma} Z^{\gamma + \epsilon} \times \sum_{0 < |v'| \leq 4 e^{3/2} Z^3 |d|^{-1/2}}\exp\left(\left(g(\gamma)+\epsilon\right) \rank_{\Z} E_{16dv'^2}(\Q)\right).
\end{align}
We now draw upon an upper bound for the rank from \cite[Proposition 2]{Fouvry}, namely
\[
\rank_{\Z} E_{16dv'^2}(\Q) \leq A + B\omega(2dv'^2) + \frac{2\log |\mathrm{Cl}_d[3]|}{\log 3},
\]
where $A$ and $B$ are absolute constants. Putting this upper bound into equation \eqref{eS3dZ} yields the estimate
\[
|S_3(d, Z)| \ll_{\epsilon, \gamma} Z^{\gamma + 2\epsilon} \times Z^3 |d|^{-\frac{1}{2}}\times |\mathrm{Cl}_d[3]|^{\frac{2(g(\gamma)+\epsilon)}{\log 3}}.
\]
Using the trivial bound \eqref{ePointwise}, we certainly have $ |\mathrm{Cl}_d[3]|^{\epsilon}\ll d^{\epsilon}$. Substituting this into \eqref{eq:cl1}, we obtain
\begin{equation}
\label{eq:cl2}
|\mathrm{Cl}_d[3]| \ll_{\epsilon, \gamma} Z^{20 \epsilon} \left(\frac{|d|^{1/2}}{Z} + Z^{1 + \gamma} \times |\mathrm{Cl}_d[3]|^{\frac{2g(\gamma)}{\log 3}}\right).
\end{equation}
To minimize this bound, take 
\[
Z = \left(|d|^{\frac{1}{2}}|\mathrm{Cl}_d[3]|^{-\frac{2g(\gamma)}{\log 3}}\right)^{\frac{1}{2 + \gamma}} + |d|^{\frac{1}{6}},
\]
where the term $|d|^{\frac{1}{6}}$ ensures that $Z$ satisfies \eqref{eq:lbZ}. With this choice of $Z$ and rescaling $\epsilon$, the bound \eqref{eq:cl2} becomes
\[
|\mathrm{Cl}_d[3]| \ll_{\epsilon, \gamma} |d|^{\frac{1}{2}(1 - {\frac{1}{2+\gamma})+\epsilon}} |\mathrm{Cl}_d[3]|^{\frac{2g(\gamma)}{(2 + \gamma)\log 3}} + |d|^{\frac{1}{6}(1 + \gamma) + \epsilon} \times |\mathrm{Cl}_d[3]|^{\frac{2g(\gamma)}{\log 3}}.
\]
Using the bound for $g(\gamma)$ in \eqref{eUseful} to control the $|d|^\epsilon$ factors and rearranging gives
\[
|\mathrm{Cl}_d[3]| \ll_{\epsilon, \gamma} |d|^{\kappa(\gamma) + 3 \epsilon} + |d|^{\lambda(\gamma) + 3 \epsilon},
\]
where
\[
\kappa(\gamma) \coloneqq \frac{\frac{1}{2}(1 - {\frac{1}{2 + \gamma})}}{1 - \frac{2g(\gamma)}{(2 + \gamma) \log 3}}
= \frac{1 + \gamma}{2(2 + \gamma - \frac{2g(\gamma)}{\log 3})} \quad \text{ and } \quad \lambda(\gamma) \coloneqq \frac{1+\gamma}{6(1 - \frac{2g(\gamma)}{\log 3})}.
\]
We can readily check that $\kappa(\gamma) > \lambda(\gamma)$ holds for all $0.016<\gamma< 1$, and that $\kappa(\gamma)$ takes its minimum value $\kappa(\gamma) \approx 0.3193075884$ when $\gamma\approx 0.4326086318$ is the unique root of 
$
G(\gamma)\coloneqq-(1+\gamma)g'(\gamma)+g(\gamma)-\frac{\log 3}{2}
$
in the interval $(0,1)$. Direct computations show that
\begin{align}
G(\gamma)&=
 \log\frac{1-\gamma}{2\sqrt{(3-\gamma)(1+\gamma)}}+
\frac{2 (2 -\gamma)}{ \sqrt{(3-\gamma)^3(1+\gamma)}}  \log\frac{2+\sqrt{(3-\gamma)(1+\gamma)} }{2-\sqrt{(3-\gamma)(1+\gamma)} }-\frac{\log 3}{2},\label{eq:Groot}\\
\kappa(\gamma)&=
\frac{\frac{1}{4}(\gamma + 1) \log 3}{(\gamma + 2) \frac{\log 3}{2} -  \log\frac{1-\gamma}{2\sqrt{(3-\gamma)(1+\gamma)}}+  \frac{1}{\sqrt{(3-\gamma)(1+\gamma)}}\log\frac{2+\sqrt{(3-\gamma)(1+\gamma)} }{2-\sqrt{(3-\gamma)(1+\gamma)} }}. \label{eq:kapform}
\end{align}
This completes the proof of Theorem~\ref{theorem:main}.

\section{Average bounds for class group torsion}
\label{sBonusT}
We will now prove Theorem \ref{t2}. Since the argument follows the same lines as Heath-Brown--Pierce \cite{HBP}, we will only indicate the necessary modifications. As in their work, we introduce a parameter $X^{\frac{1}{2\ell}} \leq Z \leq X$ to be optimized later. We consider 
$$
M(d, Z) := \sum_{Z \leq p < 2Z} \chi_d(p),
$$
where $\chi_d(p)$ is the quadratic character associated to $\Q(\sqrt{d})$. Computing a high moment of $M(d, Z)$ and applying the large sieve as in \cite[Section 2.1]{HBP} gives
$$
\left|\left\{X \leq d < 2X : |M(d, Z)| \geq \frac{Z}{4 \log Z}\right\}\right| \ll_\epsilon X^\epsilon.
$$
By the trivial bound \eqref{ePointwise}, we may ignore this small exceptional set. Then by Theorem \ref{tEV} we get
$$
\sum_{X \leq d < 2X} h_\ell(d) \ll_\epsilon X^{1/2 + \epsilon} + \frac{X^{3/2 + \epsilon}}{Z} + \frac{X^{1/2 + \epsilon}}{Z^2} \cdot \left(\sum_{X \leq d < 2X} |S_\ell(d, Z)|\right).
$$
The next result is key.

\begin{prop}
\label{pHB}
For all $\ell \geq 3$, for all $X^{\frac{1}{2\ell}} \leq Z \leq X$ and for all $\epsilon > 0$ we have
$$
\sum_{X \leq d < 2X} |S_\ell(d, Z)| \ll_{\epsilon, \ell} X^\epsilon (Z^2 X^{1/2} + Z^{\ell + 2} X^{-1/2}).
$$
\end{prop}

Note that this result is the analogue for real quadratic fields of \cite[Proposition 2.3]{HBP}. Indeed, the $S_\ell(d, Z)$ as defined in \cite[Proposition 2.1]{HBP} considers the equation $4 (p_1 p_2)^\ell = u^2 + dv^2$, while we get the equation $4 (p_1 p_2)^\ell = u^2 - dv^2$ after taking norms in the equation $\beta O_d = (\mathfrak{p}_1 \mathfrak{p}_2)^{\ell}$. 

Because of the bounds on $u$ and $v$ in $\beta = u + v \sqrt{d}$, one can now directly mimic the proof in \cite[Section 3]{HBP} to prove Proposition \ref{pHB}. The optimal choice $Z = X^{\frac{3}{2\ell + 2}}$ ends the proof of Theorem \ref{t2}.

\bibliographystyle{abbrv}
\bibliography{Torsion}
\end{document}